\documentclass[12pt,reqno]{article}
\usepackage{amssymb}
\usepackage{amsmath}
\usepackage{amsthm}
\usepackage{amsfonts}
\usepackage{tabularray}
\usepackage{float}
\usepackage{tikz}
\usetikzlibrary{arrows.meta,automata,positioning} 
\usepackage{pgfplots}
\usetikzlibrary{positioning}
\usetikzlibrary{arrows.meta,automata,positioning} 

\definecolor{lightergray}{gray}{.8}
\definecolor{lightestgray}{gray}{.9}

\usepackage[colorlinks=true,
linkcolor=webgreen,
filecolor=webbrown,
citecolor=webgreen]{hyperref}

\definecolor{webgreen}{rgb}{0,.5,0}
\definecolor{webbrown}{rgb}{.6,0,0}
\definecolor{lightergray}{gray}{.8}
\definecolor{lightestgray}{gray}{.9}

\usepackage{fullpage}
\title{How Prime Factors Form Fractals}
\author{Micah D. Tillman (mdtillman@stanford.edu)}

\begin{document}
\newcommand{\vpn}[1]{$v_{#1}(n)$}
\newcommand{\svpn}[1]{$v_{#1}\langle\rangle$}
\newcommand{\svpns}[1]{$v_{#1}\{\}$}
\newcommand{\lseq}[1]{$l\langle\rangle_{#1}$}
\newcommand{\emd}{\textemdash ~}
\newcommand{\levy}{L\'{e}vy~}
\newcommand{\oddpart}{$o_n\langle\rangle$}
\newcommand*\circled[1]{\tikz[baseline=(char.base)]{
            \node[shape=circle,draw,inner sep=2pt] (char) {#1};}}
\newcommand{\seqnum}[1]{\href{https://oeis.org/#1}{\rm \underline{#1}}}

\theoremstyle{plain}
\newtheorem{theorem}{Theorem}
\newtheorem{corollary}[theorem]{Corollary}
\newtheorem{lemma}[theorem]{Lemma}
\newtheorem{proposition}[theorem]{Proposition}

\theoremstyle{definition}
\newtheorem{definition}[theorem]{Definition}
\newtheorem{example}[theorem]{Example}
\newtheorem{conjecture}[theorem]{Conjecture}

\maketitle
\begin{abstract}
We explore a new sieve that generates both primes and prime factorizations, without resorting to division. We demonstrate that the integer sequences generated by the sieve are the $p$-adic valuations of $n$, and that each is a fractal sequence. We then show that these sequences produce geometrical fractals like the \levy Dragon. We end by showing the connection between the odd part of $n$ integer sequence and the Heighway Dragon. 
\end{abstract}

\section{Introduction}
There is a beauty in prime factorizations that few ever see. Repeated trial division is not easy to appreciate, but prime factorizations also fail to form an appreciable \textit{whole}. They feel chaotic, unpredictable; the only guarantee is that $n$'s prime factors will differ from $n+1$'s.

The sieve described below eliminates both the division and the disunity. It reveals that fractal patterns underlie the chaos, patterns that are easy to grasp and yet produce complex geometrical figures. In section 2, we see how the sieve works, demonstrating that it generates both primes and prime factorizations. In section 3, we show that the integer sequences it produces \emd the $p$-adic valuations of $n$ \emd are fractal. In section 4, we show that 2's integer sequence encodes the turns in the \levy Dragon. In section 5, we examine several other fractals encoded by $p$-adic sequences and discuss $2^3$ open questions.

\section{The Sieve}
We begin with an empty table. In the table's top row, we place the positive integers up to $m$, where $m$ is as large as we wish. We leave the leftmost cell of the row empty, however, placing a 1 in its second cell, a 2 in its third, a 3 in its fourth, etc., working to the right. 

In what follows, let the number in the topmost cell of each column be the column's \textit{header} and the column whose header is $n$ be \textit{the column for} $n$. With the table set, the sieve begins:

\subsection{Steps of the sieve}
\begin{enumerate}
	\item Find the first column to the right of the column for 1 that has a header, but whose other cells are either empty or contain 0s. If none can be found, stop. Otherwise, we designate column's header, $p$.

	\item Generate an integer sequence as follows:\footnote{Here, we generalize a method proposed by Cloitre \cite{cloitre}.}

	\begin{enumerate}
		\item Create an integer sequence that contains nothing but a single 0. 
		\item Make $p-1$ copies of the integer sequence in its current state.
		\item Append those copies to the original, end-to-end.
		\item Increase the final integer in the resulting sequence by 1.
		\item If the sequence contains fewer than $m$ terms, return to (b). Otherwise, go to step 3.
	\end{enumerate}

	\item Find the topmost empty row in the table. Place $p$ in its leftmost cell. Place the integer sequence from step 2 in the cells to the right of $p$, such that the first integer in the sequence is placed in the cell just to the right of $p$, the second is placed two cells to the right, the third is placed three cells to the right, etc.

	\item Return to step 1.

\end{enumerate}

\subsection{Illustrating the sieve}
By way of illustration, imagine that we start with the following table, where we have limited ourselves to an $m$ of 10:

\begin{table}[H]
\centering
\begin{tblr}{rowspec={|Q[c,lightgray]|Q[c]}, colspec={|Q[c,lightgray]|}}
\SetCell{bg=white}~	&	1	&	2	&	3	&	4	&	5	&	6	&	7	&	8	&	9	&	10\\
					&	 	&	 	&	 	&	 	&	 	&	 	&	 	&	 	&	 	&\\
\end{tblr}
\caption{A partial view of the table for displaying the sieve's results.}
\end{table}

In step 1, we identify the column for 2 as the first (after the column for 1) that has a header but whose other cells are all empty, or contain only 0s. Its header is 2, so $p=2$.

In step 2, we generate an integer sequence for 2 as follows:

\begin{enumerate}
\item[]
\begin{enumerate}
\item We create an integer sequence that contains nothing but a single 0: $\langle0\rangle$
\item We make $p-1$ (i.e., $2-1=1$) copies of the integer sequence in its current state: $\langle0\rangle$
\item We append those copies to the original, end-to-end: $\langle0, 0\rangle$
\item We increase the final integer in the resulting sequence by 1: $\langle0, 1\rangle$
\item Since our integer sequence has fewer than $m=10$ terms, we return to (b).
\item[(b)] We make $p-1$ (i.e., $2-1=1$) copies of the integer sequence in its current state: $\langle0, 1\rangle$
\item[(c)] We append those copies to the original, end-to-end: $\langle0, 1, 0, 1\rangle$
\item[(d)] We increase the final integer in the resulting sequence by 1: $\langle0, 1, 0, 2\rangle$
\item[(e)] Since our integer sequence has fewer than $m=10$ terms, we return to (b) \ldots
\end{enumerate}
\end{enumerate}

In step 3, we would then place $p$ (in this case, 2) in the first cell of the first empty row, and place the integer sequence in the cells to its right.
 
\begin{table}[H]
\centering
\begin{tblr}{rowspec={|Q[c,lightgray]|Q[c]}, colspec={|Q[c,lightgray]|}}
\SetCell{bg=white}	&	1	&	2	&	3	&	4	&	5	&	6	&	7	&	8	&	9	&	10\\
				2	&	0 	&	1 	&	0 	&	2 	&	0 	&	1	&	 0	&	3 	&	0 	&	1\\
\end{tblr}
\caption{Partial table after one iteration.}
\end{table}

Then, we would repeat the process. In step 1, we identify the column for 3 as the first (to the right of the column for 1) that has a header but whose cells are all empty, or contain only 0s. The number in its topmost cell is 3, so $p=3$.

\begin{enumerate}
\item[]
\begin{enumerate}
\item We create an integer sequence that contains nothing but a single 0: $\langle0\rangle$
\item We make $p-1$ (i.e., $3-1=2$) copies of the integer sequence in its current state: $\langle0\rangle\langle0\rangle$
\item We append those copies to the original, end-to-end: $\langle0, 0, 0\rangle$
\item We increase the final integer in the resulting sequence by 1: $\langle0, 0, 1\rangle$
\item Since our sequence has fewer than $m=10$ terms, we return to (b).
\item[(b)] We make $p-1$ (i.e., $3-1=2$) copies of the integer sequence in its current state: $\langle0, 0, 1\rangle\langle0, 0, 1\rangle$
\item[(c)] We append those copies to the original, end-to-end: $\langle0, 0, 1, 0, 0, 1, 0, 0, 1\rangle$
\item[(d)] We increase the final integer in the sequence by 1: $\langle0, 0, 1, 0, 0, 1, 0, 0, 2\rangle$
\item[(e)] Since our sequence still has fewer than $m=10$ terms, we return to (b) \ldots
\end{enumerate}
\end{enumerate}

We would then place $p$ (in this case, 3) in the first cell of the first open row, and place the integer sequence in the cells to its right.
 
\begin{table}[H]
\centering
\begin{tblr}{rowspec={|Q[c,lightgray]|Q[c]|Q[c]}, colspec={|Q[c,lightgray]|}}
\SetCell{bg=white}	&	1	&	2	&	3	&	4	&	5	&	6	&	7	&	8	&	9	&	10\\
				2	&	0 	&	1 	&	0 	&	2 	&	0 	&	1	&	 0	&	3 	&	0 	&	1\\
				3	&	0 	&	0 	&	1 	&	0 	&	0 	&	1 	&	 0	&	0 	&	2 	&	0\\
\end{tblr}
\caption{Partial table after two iterations.}
\end{table}

On the next iteration, we identify the column for 5 as the first (to the right of the column for 1) that has a header but whose cells are all empty, or contain only 0s. The number in its topmost cell is 5. We generate the sequence for 5 and add it to the table.
 
We claim that if we continue the process, the left column of the table will eventually contain every prime $p\le m$, and only those primes. We claim, further, that the columns of the table would \emd when read properly \emd give us the prime factorization of every positive integer $n\le m$ (where $n>1$). We will prove both claims in section 2.4.

A slightly larger segment of the table is provided here, for reference:

\begin{table}[H]
\centering
\begin{tblr}{colsep=4pt, rowspec={|Q[c,lightgray]|Q[c]|Q[c]|Q[c]|Q[c]|Q[c]|Q[c]|Q[c]}, colspec={|Q[c,lightgray]| Q[c,1em] Q[c,1em] Q[c,1em] Q[c,1em] Q[c,1em] Q[c,1em] Q[c,1em] Q[c,1em] Q[c,1em] Q[c] Q[c] Q[c] Q[c] Q[c] Q[c] Q[c]}}
\SetCell{bg=white}	&	1	&	2	&	3	&	4	&	5	&	6	&	7	&	8	&	9	&	10	&	11	&	12	&	13	&	14	&	15	&	16	&	$\cdots$ 	\\
2					&	0	&	1	&	0	&	2	&	0	&	1	&	0	&	3	&	0	&	1	&	0	&	2	&	0	&	1	&	0	&	4	&	$\cdots$ 	\\
3					&	0	&	0	&	1	&	0	&	0	&	1	&	0	&	0	&	2	&	0	&	0	&	1	&	0	&	0	&	1	&	0	&	$\cdots$ 	\\
5					&	0	&	0	&	0	&	0	&	1	&	0	&	0	&	0	&	0	&	1	&	0	&	0	&	0	&	0	&	1	&	0	&	$\cdots$ 	\\
7					&	0	&	0	&	0	&	0	&	0	&	0	&	1	&	0	&	0	&	0	&	0	&	0	&	0	&	1	&	0	&	0	&	$\cdots$ 	\\
11					&	0	&	0	&	0	&	0	&	0	&	0	&	0	&	0	&	0	&	0	&	1	&	0	&	0	&	0	&	0	&	0	&	$\cdots$ 	\\
13					&	0	&	0	&	0	&	0	&	0	&	0	&	0	&	0	&	0	&	0	&	0	&	0	&	1	&	0	&	0	&	0	& 	$\cdots$ 	\\
$\vdots$	&	$\vdots$	&	$\vdots$	&	$\vdots$	&	$\vdots$	&	$\vdots$	&	$\vdots$	&	$\vdots$	&	$\vdots$	&	$\vdots$	&	$\vdots$	&	$\vdots$	&	$\vdots$	&	$\vdots$	&	$\vdots$	&	$\vdots$	&	$\vdots$	&	$\ddots$ 	\\
\end{tblr}
\caption{Partial table after six iterations.}
\end{table}

\subsection{Reading the table}

In what follows, let the number in the leftmost cell of each row be the row's \textit{header}, and the row whose header is $p$ be \textit{the row for }$p$.

To find a given positive integer's prime factorization, locate the column for that number in the table. Then, identify all cells in the same column, below its header, that contain a number greater than 0. For each such cell, identify the header of the row to which it belongs, and raise that header to the number contained within the cell. The prime factorization of the column's header will be the product of those row headers raised to those numbers.

To find the prime factorization of 24, for example, we first locate its column. 
\begin{table}[H]
\centering
\begin{tblr}{rowspec={|Q[c,lightgray]|Q[c]|Q[c]|Q[c]|Q[c]|Q[c]}, colspec={|Q[c,lightgray]}}
\SetCell{bg=white}	&	$\cdots$		&	24	&	$\cdots$ 	\\
2					&	$\cdots$		&	3	&	$\cdots$ 	\\
3					&	$\cdots$		&	1	&	$\cdots$ 	\\
5					&	$\cdots$		&	0	&	$\cdots$ 	\\
7					&	$\cdots$		&	0	&	$\cdots$ 	\\
$\vdots$	&	$\vdots$	&	$\vdots$	&	$\ddots$ 	\\
\end{tblr}
\caption{Reading the prime factorization of 24.}
\end{table}
\noindent
We then see that only the second and third cells in the column contain a number larger than 0. There is a 3 in the row for 2, and a 1 in the row for 3. Thus, we raise 2 to the third power and 3 to the first power, then multiply the two together: $24=2^3\times3^1$.

\subsection{Does the sieve work?}
Now that we know how to run the sieve, we need to know how confident we should be that it does what it promises. Does it identify all and only the primes? Does it provide the complete prime factorization of every positive integer $>1$?

\begin{proposition}
If steps 2 and 3 in the sieve eliminate as composite all and only the multiples of each number that the sieve identifies as prime, then the sieve identifies all and only the primes as prime.  
\end{proposition}

\begin{proof}
Our sieve follows Eratosthenes in its manner of choosing the next number to designate as prime: it finds the first positive integer greater than 1 that has yet to be eliminated as composite or designated as prime. The first such number is 2. It then designates the number it has just identified as prime by placing it in a cell in the first column (rather than by circling it, for example). Finally, as we will see below, it eliminates all and only the multiples of that prime by placing an integer $>0$ into a cell in every column corresponding to a multiple thereof (rather than by crossing off every multiple, for example).

Assuming all of this is correct, if \emd as we work from left to right, skipping the column for 1 \emd we arrive at a column that contains no numbers $>1$ below its header, we have arrived at a number that is not a multiple of the positive integers less than it (1 excepted). But that means the number must be prime, so the sieve will designate it as such and eliminate all its multiples. 

Therefore, \textit{if} the algorithm for generating integer sequences actually eliminates as composite all and only the multiples of each prime that the sieve identifies, then the sieve identifies all and only primes as prime. 
\end{proof}

\begin{proposition}
Steps 2 and 3 in the sieve eliminate as composite all and only the multiples of the numbers that the sieve identifies as prime.
\end{proposition}

\begin{proof}
Every sequence in step 2 of the sieve starts as a single 0, which is expanded in the first iteration to become a sequence of $p-1$ consecutive 0s, followed by a 1 at index $p$. Future iterations will expand the sequence only by making copies of what they are given and appending those copies end-to-end. So, the second iteration will generate a sequence consisting of $p$ subsequences, each of whose $p$th term is a 1. It will then increment the final of those 1s, converting it into a 2. 

The same basic thing will happen again with each following iteration. Copies of a sequence in which every $p$th term is $>0$ will be made, all copies will be concatenated together with the original, and the last of those terms will be made even larger than it was before. Meanwhile, all the terms whose indexes are not a multiple of $p$ will remain 0s.

That sequence of integers will then be placed into the table, with the $n$th term in the sequence being placed in the column for $n$. Since every $p$th term in the sequence will be $>0$, every $p$th number in the table's first row will end up with a number $>0$ placed below it. This will mark the headers of the form $h=k\times p$ (where $k$ is a positive integer) \emd all multiples of $p$ \emd as composite, but will not mark any others as composite. Therefore, all and only the multiples of the primes will be identified as composite.
\end{proof}

\begin{lemma}
The sieve identifies all and only the primes as prime.
\end{lemma}

\begin{proof}
This follows from Propositions 1 and 2, by \textit{modus ponens}.
\end{proof}

So, our sieve sieves. But it is also supposed to generate prime factorizations. If it is going to do that, it will need to do more than just putting \textit{some} number $>0$ in the right columns. It will need to put the \textit{right} number $>0$ in the right columns.

\begin{proposition}
If $n$'s prime factorization contains the multiplicand $p^x$, step 3 in the sieve places $x$ in the table's column for $n$ and row for $p$. Otherwise it places a 0 in the table's column for $n$ and row for $p$. 
\end{proposition}

\begin{proof}
Because step 2 in the sieve starts with a sequence of length 1, and then multiplies the length of the sequence by $p$ each time, the index of the term it increments will always be $p^j$, where $j$ is the number of the current iteration. After the incrementing step in a given iteration, furthermore, the $p^j$th term will always be $j$. But the $p^j$th term is supposed to represent the exponent of $p$ in the prime factorization of $p^j$, which just \textit{is} $j$. Therefore, $j$ is exactly what we want at index $p^j$.

\begin{table}[H]
\centering
\begin{tblr}{rowspec={|Q[c]|Q[c]|Q[c]|}, colsep=5pt, colspec={|Q[r]|}}
 $n$: & 1 & 2 & 3 & 4 & 5 & 6 & 7 & 8 & 9 & 10 & $\cdots$ 	\\
Fact.: & n/a 	&	$2^1$ 	&	$3^{\circled{\scriptsize 1}}$	&	$2^2$ 	&	$5^1$	&	$2^1\times3^{\circled{\scriptsize 1}}$	&	$7^1$	&	$2^3$	&	$3^{\circled{\scriptsize 2}}$	&	$2^1\times5^1$ 	&	$\cdots$ 	\\
Seq.: &	0	& 0 & \circled{1} & 0	& 0 &	\circled{1} & 0	& 0 &	\circled{2}	&	0	&	$\cdots$ 	\\
\end{tblr}
\caption{How $j$ (from $3^j$ in prime factorizations) appears in the integer sequence for 3.}
\end{table}

So far, so good. But when we create copies of the sequence during the next iteration, we end up placing $p-1$ copies of $j$ at \textit{later} indexes as well.

\begin{table}[H]
\centering
\begin{tblr}{rowspec={|Q[c]|Q[c]|Q[c]|Q[c]|}, colspec={|Q[r]|Q[l]|}}
Start: 		& 0\\
Duplicate:	& 0, 0, 0\\	
Increment:	& 0, 0, \circled{1}\\
Duplicate:	& 0, 0, 1, 0, 0, \circled{1}, 0, 0, \circled{1}\\
\end{tblr}
\caption{How $j$ is copied by later terms, when $p=3$.}
\end{table}

Those additional $j$ terms will fall at the ends of the copies of the sequence, each of which has the same length as the current sequence. This means the first additional $j$ will fall at index $2\times p^j$, the second will fall at index $3\times p^j$, etc., up to the $(p-1)$th additional copy of $j$ falling at index $p\times p^j$. The term at index $o\times p^j$, where $o<p$ is a positive integer, is supposed to represent the exponent of $p$ in the prime factorization of the number that just \textit{is} $o\times p^j$. But the exponent of $p$ in that prime factorization will just be $j$ again. Therefore, $j$ is exactly the number we want at that index.

Things are different for the term at index $p\times p^j$, however. Since $p\times p^j=p^{j+1}$, we want $p+1$ at that index. But that is exactly what we will get, since that is the term we will increment at the end of the iteration. So, if we want to end up with $j+1$ after the incrementing step, we need to place a $j$ there first. This is what step 2 of the sieve does.

Therefore, step 2 of the sieve will place an $x$ at the $n$th index in the integer sequence for $p$ if $n$'s prime factorization contains the multiplicand $p^x$. If $n$'s prime factorization does not contain a multiplicand whose base is $p$, step 2 will place a 0 at index $n$ instead. 

Finally, step 3 will place the integer sequence in the row for $p$, with the term at index $n$ being in the column for $n$. Therefore, if $n$'s prime factorization contains the multiplicand $p^x$, step 3 in the sieve places $x$ in the table's column for $n$ and row for $p$. Otherwise it places a 0 in the table's column for $n$ and row for $p$. 
\end{proof}

\begin{lemma}
The sieve provides the prime factorization of every integer $>1$.
\end{lemma}

\begin{proof}
Each prime by which a positive integer $n$ is divisible will appear as the base in one of the multiplicands in the prime factorization of $n$. Each of those primes, furthermore, will be raised to some positive integer in $n$'s prime factorization. By Proposition 4, if $n$'s prime factorization contains the multiplicand $p^x$, step 3 in the sieve places $x$ in the table's column for $n$ and row for $p$. From this, it is possible to determine which primes are the bases of multiplicands in the prime factorization of $n$ and what the exponents of those primes are. Therefore, the sieve tells us what all the components of $n$'s prime factorization are, and how to combine them. Therefore, the sieve provides the prime factorization of every integer $>1$.
\end{proof}

\begin{theorem}
The sieve identifies all and only primes as prime and provides the prime factorization of every integer $>1$.
\end{theorem}

\begin{proof}
This follows from Lemmas 3 and 5, by conjunction.
\end{proof}

\subsection{Discussion}
With this, the first half of our task is complete. We can obtain both primes and prime factorizations without trial division. It would be reasonable to ask, however, whether the gain was worth the cost. The young math student who learns this sieve may avoid their greatest nemesis \emd long division \emd but in place of that drudgery we give them the tedium of repeated copying-and-pasting.

There is a difference between these two sorts of toil, however. Until the young mathematician is allowed the use of a calculator, division requires actual \textit{thought}. Our sieve, in contrast, requires only mechanical counting.  

The sieve is more than a labor-saving device, however. In removing the obstacle of calculation, it reveals the fractal structures underlying the natural numbers. It will be our task in the next section to demonstrate this.

\section{Fractal Sequences}

The integer sequences described above are not new. The sequence for 2 is the ``2-adic valuation of $n$'' \emd sequence \seqnum{A007814} in the OEIS \cite{oeis}. Three's sequence, furthermore, is the ``3-adic valuation of $n$'' (\seqnum{A007949}), and 5's is the ``5-adic valuation of $n$'' (\seqnum{A112765}). There is significant danger of terminological confusion here, however, so it will behoove us to be annoyingly precise. 

\begin{definition} Let \vpn{p}, \svpn{p}, and \svpns{p} be defined as follows:
\begin{itemize}
\item \vpn{p} (``the $p$-adic valuation of $n$''): a function that returns the largest integer $k$ (for some prime $p$ and positive integer $n$) such that $p^k$ divides $n$. 
\item \svpn{p} (``the $p$-adic valuation sequence''): the integer sequence in which the $i$th term is the the number returned by $v_p(n_i)$.
\item \svpns{p} (``the $p$-adic valuation set''): the set of all \svpn{p} sequences. 
\end{itemize}

\end{definition}

\noindent The function \vpn{p} plays an important role in $p$-adic number systems, in that it is used to define $p$-adic \textit{absolute} values.  Our focus here, however, will be on the fact that it generates fractals.

\subsection{Requirements for being fractal}

The 2-adic, 3-adic, and 5-adic valuation sequences are just three among the many given fractal-associated labels in the OEIS. To my knowledge, however, no one has provided a unified account of the fractal nature of all the $p$-adic valuation sequences, or offered a \textit{demonstration} that they are all fractal. 

To begin, we need a definition of ``fractal sequence.''

\begin{definition} 
A sequence is fractal if it is predictably self-containing and aperiodic.
\end{definition}

That fractals tend to contain copies of themselves, and thus tend to be ``self-containing,'' is a commonplace. What it would mean for an integer sequence to contain itself may not be obvious at first, however, and what it would mean for a sequence to be ``predictably'' self-containing, even more so. Following Kimberling \cite[p.\ 2]{kimberling22},\footnote{Kimberling \cite{kimberling95} may have offered the earliest definition of fractal sequences, as well as the first distinction between self-containing and fractal sequences \cite{kimberling97}, though his most recent definition of fractal sequences \cite[p.\ 4]{kimberling22} differs from the one with which we will work here.} we will work with the following definition of self-containment:

\begin{definition}
A sequence, $a_n$, is \textit{self-containing} if it has a proper subsequence, $a_{n_i}$, such that
$a_{n_i}=a_i \text{~for all~} i \in \mathbb{N} = \{1, 2, 3,\ldots\}.$
\end{definition}

For a sequence to be self-containing in the structured way we expect of a fractal, however, we need a reason for selecting $a_j$ to serve as the $i$th term in $a_{n_i}$, other than the fact that $j>i$ and $a_j$ happens to equal $a_i$. If we just scan through the sequence till we find the next term that will work for our subsequence, we will have shown that we can impose a pattern on our sequence, but not that we have discovered a genuine pattern therein. What we want, then, is a function, algorithm, or rule that will tell us \textit{where} to find the next term for our subsequence, not a \textit{post hoc} criterion to tell us \textit{that} we've found a term that will work once we get to it.

To capture this requirement, we will limit ourselves to ``predictably'' self-containing sequences, defined as follows:

\begin{definition}
A sequence, $a_n$, is \textit{predictably self-containing} if it is self-containing and there is at least one function or algorithm that enables us to determine (``predict'') which term in $a_n$ should be selected to serve as the $i$th term of $a_n$'s identical subsequence without requiring us to know the value of $a_i$.
\end{definition}

Now for aperiodicity. Talk of ``fractals'' becomes useful only when we are confronted with irregularity or roughness that ``zooming in'' does not smooth out (see Mandelbrot \cite[p.\ 1]{mandelbrot} and Stewart \cite[p.\ 4--5]{stewart}). There are self-containing figures and sequences, however, that are regular from the start. The example that comes most readily to mind is the straight line (see \levy \cite[p.\ 186]{levy}), but infinitely-long periodic sequences like $\langle 0,$ 0, 0, 0, 0, 0, 0, 0, $\ldots\rangle$ or $\langle 0,$ 1, 2, 3, 0, 1, 2, 3, $\ldots\rangle$ are much the same. 

To be justified in calling a sequence ``fractal,'' then, we will need evidence that it is aperiodic in the following sense:

\begin{definition}
Let ``$x_n + y_n$'' denote the concatenation of sequences $x_n$ and $y_n$, such that $y_n$ is appended to $x_n$ to form a longer sequence. We say a sequence, $a_n$, is \textit{aperiodic} if there is no shorter sequence, $b_n$, such that $a_n = b_n + b_n + b_n + \cdots$. 
\end{definition} 

So, to show that a sequence is fractal we need three things: 

\begin{enumerate}
\item a way to locate (pick out, identify) within $a_n$ the $i$th term of the subsequence, $a_{n_i}$, without making reference to the value of $a_i$, 
\item proof that the subsequence, $a_{n_i}$, is identical with $a_n$, and 
\item proof that there is no sequence $b_n$, shorter than $a_n$, such that $a_n$ = $b_n + b_n + b_n + \cdots$.
\end{enumerate}

\subsection{The subsequence location requirement}
Contributors to the OEIS tend to offer one or both of two types of evidence that a sequence is fractal. The first is what we might call a \textit{transformation} rule. Kimberling's ``lower trim'' operation \cite[p.\ 163]{kimberling97}, for example, involves subtracting 1 from all the members of the sequence, then discarding any members that have reached a predefined lower bound (e.g., $0$ or $-1$). 

The second sort of evidence one usually finds in the OEIS (see Gilleland \cite{gilleland}) is what we follow Faye \cite{faye} in calling a ``decimation rule.'' This term seems to be borrowed from signal processing (see Taylor \cite[p.\ 747]{taylor} and F\'{u}ster-Sabater and Cardell \cite[p.\ 79]{fuster}) and refers to a method for deciding which members of a sequence, $a_n$, are to be considered as belonging to a subsequence, $a_{n_i}$. 

Since a transformation rule alters the values of terms in a sequence, it fails to demonstrate that the original sequence was self-containing. Decimation rules, in contrast, simply select members of the original sequence for inclusion in a subsequence. For our purposes, then, we need to work out an appropriate decimation rule if we want to show that the various members of \svpns{p} are fractal.

\subsubsection{The decimation rule}
Let \textit{the remainder of the sequence} refer to the terms in \svpn{p} that have yet to be evaluated for inclusion or exclusion from the subsequence. With that meaning in mind, consider the following three-step decimation rule:

\begin{enumerate}
\item Exclude the first $p$ terms in the remainder of the sequence from membership in the subsequence.
\item Include the next term of the sequence as a member of the subsequence.
\item Go to step 1.
\end{enumerate}

We claim that applying this decimation rule to \svpn{p} will identify a proper subsequence of \svpn{p} that is identical to \svpn{p}. For \svpn{2}, the decimation rule would produce:

\begin{table}[H]
\centering
\begin{tblr}{rowspec={Q[c]Q[c]}, colspec={Q[r]Q[l,27em]}, colsep=5pt}
Original: & {0, 1, 0, 2, 0, 1, 0, 3, 0, 1, 0, 2, 0, 1, 0, 4, 0, 1, 0, 2, 0, 1, 0, 3,\\ 0, 1, 0, 2, 0, 1, 0, 5, 0, 1, 0, 2, 0, 1, 0, 3, 0, 1, 0, 2, 0, 1, 0, 4, $\ldots$}\\
Decimated: & {{\color{lightgray}0, 1,} 0, {\color{lightgray}2, 0,} 1, {\color{lightgray}0, 3,} 0, {\color{lightgray}1, 0,} 2, {\color{lightgray}0, 1,} 0, {\color{lightgray}4, 0,} 1, {\color{lightgray}0, 2,} 0, {\color{lightgray}1, 0,} 3,\\ {\color{lightgray}0, 1,} 0, {\color{lightgray}2, 0,} 1, {\color{lightgray}0, 5,} 0, {\color{lightgray}1, 0,} 2, {\color{lightgray}0, 1,} 0, {\color{lightgray}3, 0,} 1, {\color{lightgray}0, 2,} 0, {\color{lightgray}1, 0,} 4, $\ldots$}
\end{tblr}
\caption{The decimation rule applied to \svpn{2}.}
\end{table}

\noindent For \svpn{3} it would produce:

\begin{table}[H]
\centering
\begin{tblr}{rowspec={Q[c]Q[c]}, colspec={Q[r]Q[l,27em]}}
Original: & {0, 0, 1, 0, 0, 1, 0, 0, 2, 0, 0, 1, 0, 0, 1, 0, 0, 2, 0, 0, 1, 0, 0, 1, 0, 0, 3, 0, 0, 1, 0, 0, 1, 0, 0, 2, 0, 0, 1, 0, 0, 1, 0, 0, 2, 0, 0, 1, 0, 0, 1, 0, 0, 3, 0, 0, 1, 0, 0, 1, 0, 0, 2, 0, 0, 1, 0, 0, 1, 0, 0, 2, 0, 0, 1, 0, 0, 1, 0, 0, 4, 0, 0, 1, 0, 0, 1, 0, 0, 2, 0, 0, 1, 0, 0, 1, 0, 0, 2, 0, 0, 1, 0, 0, 1, 0, 0, 3, $\ldots$}\\
Decimated: & 
{{\color{lightgray}0, 0, 1,} 0, {\color{lightgray}0, 1, 0,} 0, {\color{lightgray}2, 0, 0,} 1, {\color{lightgray}0, 0, 1,} 0, {\color{lightgray}0, 2, 0,} 0, {\color{lightgray}1, 0, 0,} 1, {\color{lightgray}0, 0, 3,} 0, {\color{lightgray}0, 1, 0,} 0, {\color{lightgray}1, 0, 0,} 2, {\color{lightgray}0, 0, 1,} 0, {\color{lightgray}0, 1, 0,} 0, {\color{lightgray}2, 0, 0,} 1, {\color{lightgray}0, 0, 1,} 0, {\color{lightgray}0, 3, 0,} 0, {\color{lightgray}1, 0, 0,} 1, {\color{lightgray}0, 0, 2,} 0, {\color{lightgray}0, 1, 0,} 0, {\color{lightgray}1, 0, 0,} 2, {\color{lightgray}0, 0, 1,} 0, {\color{lightgray}0, 1, 0,} 0, {\color{lightgray}4, 0, 0,} 1, {\color{lightgray}0, 0, 1,} 0, {\color{lightgray}0, 2, 0,} 0, {\color{lightgray}1, 0, 0,} 1, {\color{lightgray}0, 0, 2,} 0, {\color{lightgray}0, 1, 0,} 0, {\color{lightgray}1, 0, 0,} 3, $\ldots$}
\end{tblr}
\caption{The decimation rule applied to \svpn{3}.}
\end{table}

The decimation rule seems to work, but we need more than a visual scan of a few terms to be sure. 

\begin{lemma}
Every sequence, \svpn{p}, is predictably self-containing.
\end{lemma}

\begin{proof}
By Definition 7 (and Proposition 4), the $m$th term in \svpn{p} is the largest integer $k$ such that $p^k$ divides $m$. The decimation rule selects for inclusion all those terms at indexes that are multiples of $p+1$. If $m=f \times (p+1)$ (where $f$ is a positive integer), the $m$th term in \svpn{p} will be the sum of the largest $r$ and $s$ such that $p^{r}$ divides $f$ and $p^{s}$ divides $(p+1)$. But $p$ cannot be a factor of $(p+1)$. Therefore, the $m$th term will just be the largest $r$ such that $p^{r}$ divides $f$. 

This means that the term at index $m=f \times (p+1)$ in \svpn{p} will be identical to the term at index $f$. So, the subsequence consisting of all terms whose indexes have the form $m=f \times (p+1)$ is identical to the entire sequence. That is, $a_{n_f}=a_f$. Therefore, by Definition 9, every sequence in \svpns{p} is self-containing. But the decimation rule does not require us to know the value of $a_i$ in order to select the $i$th term of $a_{n_i}$. Therefore, by Definition 10, every sequence in \svpns{p} is predictably self-containing.
\end{proof}

\subsection{The aperiodicity requirement} All that remains, then, is to show that \svpn{p} is aperiodic. 

\begin{lemma}
Every sequence, \svpn{p}, is aperiodic.
\end{lemma}

\begin{proof}
To show that \svpn{p} was \textit{periodic}, we would need to show that we could find an index $q$ such that if we split the sequence at $q$, copies of the subsequence consisting of the first through $q$th terms could be appended together to generate the entire sequence. But in the sequences generated by the sieve above, there can be no such $q$. As step 2 of the sieve moves from iteration to iteration, it increments the final number in the sequence. But this means that the longer the sequence grows, the larger its largest term becomes. This means, in turn, that there will always be an infinite number of terms in the complete sequence that are larger than (not the same as) any term in the subsequence that starts with the first term and ends with the $q$th, for finite $q$. Therefore, every sequence, \svpn{p}, is aperiodic.
\end{proof}

\begin{theorem}
Every sequence in \svpns{p} is fractal.
\end{theorem}

\begin{proof}
By Definition 8, a sequence is fractal if it is both predictably self-containing and aperiodic. By Lemma 12, every sequence in \svpns{p} is predictably self-containing. By Lemma 13, every sequence in \svpns{p} is aperiodic. Therefore, every sequence in \svpns{p} is fractal.
\end{proof}

\subsection{Discussion}
The exponent of a given prime in the prime factorization of any positive integer $>1$ is determined by a fractal integer sequence. That is, whether a given prime is present in a given prime factorization, and ``how much'' of it is present, follows a fractal pattern. The steadily-increasing sequence of natural numbers is the product of the interweaving of an infinite number of primes, each fluctuating according to its own fractal plan.

The talk of ``fractals'' is more than a metaphor, furthermore. At some level, the \svpn{p} sequences share the same structure as curves that are fractal in the original, geometrical sense. It will be our task in the next section to demonstrate this for \svpn{2} specifically.

\section{The 2-adic valuation sequence and the \levy Dragon}
We tend to think of the natural numbers as distributed at regular intervals along a horizontal line. If we mark each integer with a dot instead of a tick, we obtain the following:

\begin{figure}[H]
\centering
\includegraphics{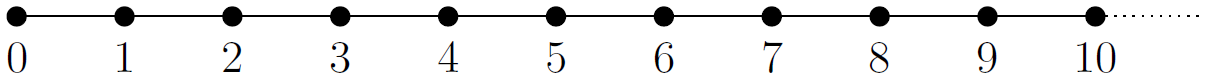}
\caption{The number line as a ``tree.''}
\end{figure}
\noindent Presented in this way, the number line becomes something like a rooted tree, where the (equal) length of each edge represents the fact that all edges have a weight of 1. Nothing particular seems to be represented by the angle between consecutive edges on a number line, however. We \textit{try} to draw number lines straight, but the world did not come to an end when Ulam drew his spiral. So, what if we put the angle between successive number line segments to use? 

Say, for example, that we wanted to represent \svpn{2} with the number line. Each number along the line (except 0) would have an associated term in the sequence. Imagine that we represent that term with a ``turn'' at that point. For example, $v_2(1)$ is 0, so the number line might turn $0\times90$ degrees to the left at 1. The value of $v_2(2)$ is 1, however, so the number line might turn $1\times90$ degrees to the left at 2. Furthermore, $v_2(4)$ is 2, so the number line might turn $2\times90$ degrees to the left at 4. 

If we interpret \vpn{n} as the number of counterclockwise 90-degree turns to make at $n$, the number line would take on the following shape:

\begin{figure}[H]
\centering
\includegraphics{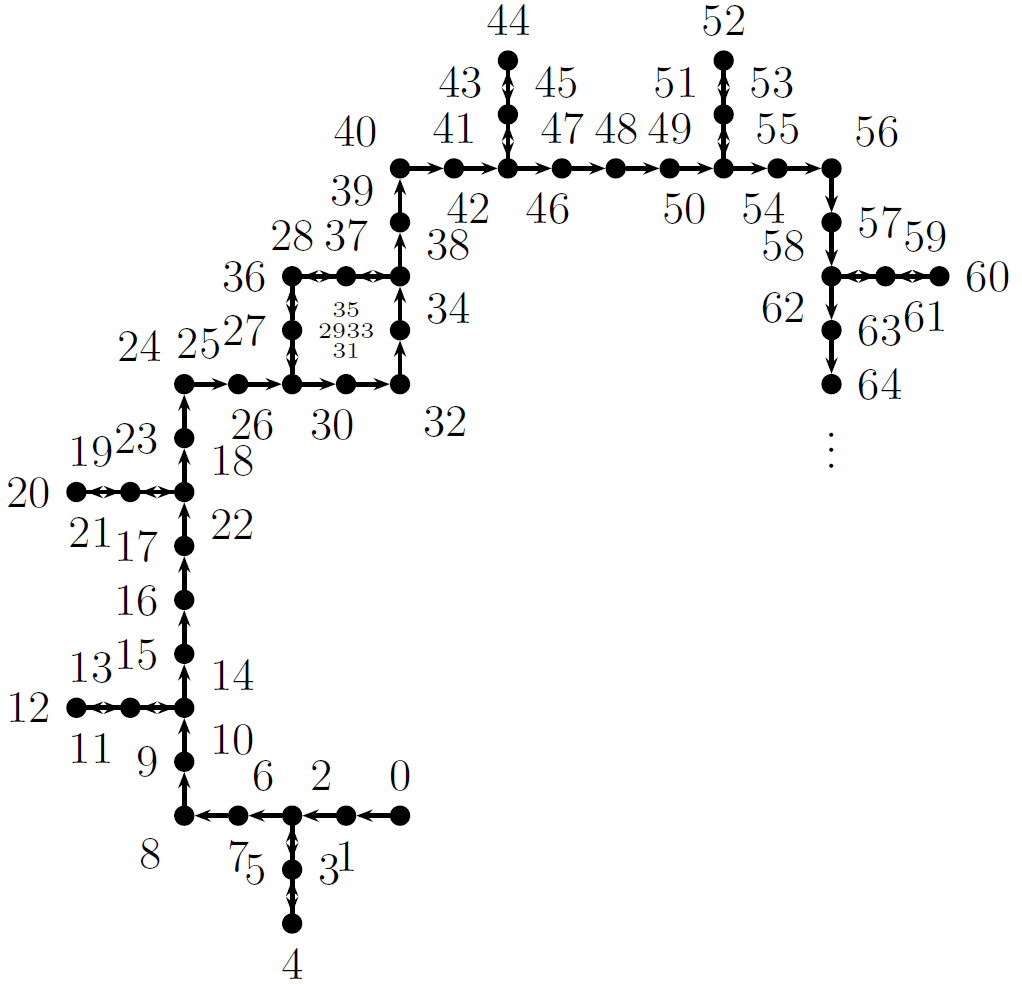}
\caption{The number line, with 90 degree turns representing \svpn{2}.}
\end{figure}

This is just the beginning of the number line, however. If we could ``zoom out'' to take in the entire number line in this configuration, we would see something like this:

\begin{figure}[H]
\centering
\includegraphics[scale=1]{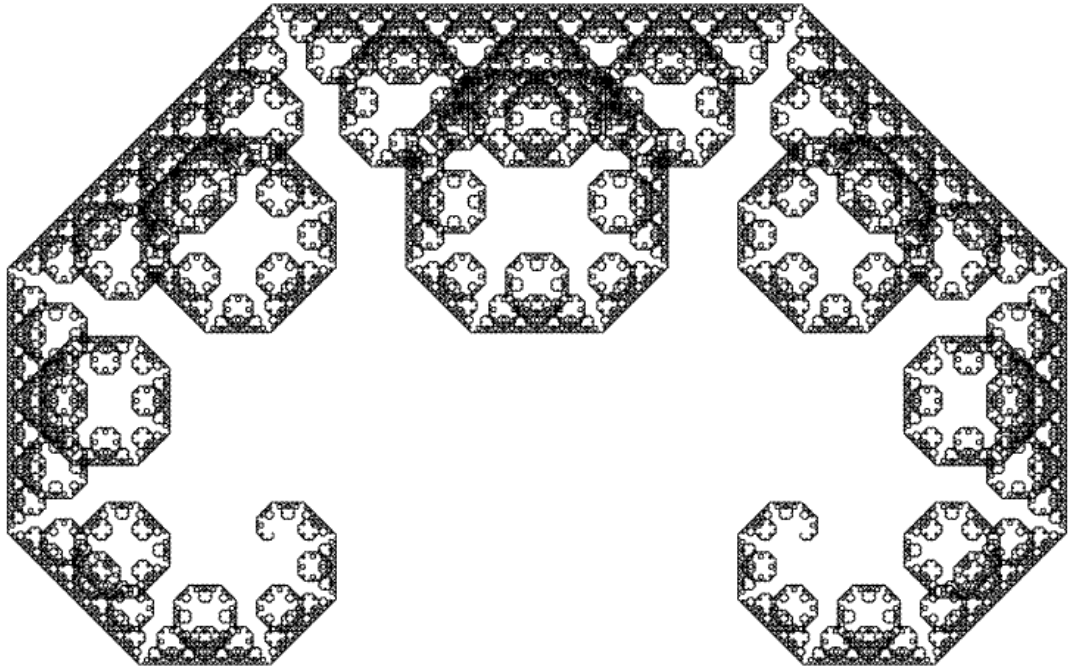}
\caption{The number line (without numbers) with 90-degree turns representing \vpn{2}.}
\end{figure}

\noindent The reader may recognize this as the \levy Dragon fractal (or ``\levy C Curve''). 

\subsection{Two sequences produce the same curve}

It is initially surprising that \svpn{2} should generate a well-known fractal curve whose definition has nothing to do with prime factorizations. A search of the OEIS, however, reveals sequence \seqnum{A346070}, entitled, ``Symbolic code for the corner turns in the \levy dragon curve.'' Its first few terms are as follows:

\begin{table}[H]
\centering
\begin{tblr}{colspec={Q[c, 31em]}}
0, 1, 0, 2, 0, 1, 0, 3, 0, 1, 0, 2, 0, 1, 0, \circled{0}, 0, 1, 0, 2, 0, 1, 0, 3, 0, 1, 0, 2, 0, 1, 0, \circled{1}, 0, 1, 0, 2, 0, 1, 0, 3, 0, 1, 0, 2, 0, 1, 0, \circled{0}, 0, 1, 0, 2, 0, 1, 0, 3, 0, 1, 0, 2, 0, 1, 0, \circled{2}, 0, 1, 0, 2, 0, 1, 0, 3, 0, 1, 0, 2, 0, 1, 0, \circled{0}, 0, 1, 0, 2, 0, 1, 0, \ldots
\end{tblr}
\caption{The terms of \seqnum{A346070} (terms differing from \svpn{2} circled).}
\end{table}

\noindent Excepting the circled terms, this sequence appears to be identical to \svpn{2}:

\begin{table}[H]
\centering
\begin{tblr}{colspec={Q[c, 31em]}}
0, 1, 0, 2, 0, 1, 0, 3, 0, 1, 0, 2, 0, 1, 0, \circled{4}, 0, 1, 0, 2, 0, 1, 0, 3, 0, 1, 0, 2, 0, 1, 0, \circled{5}, 0, 1, 0, 2, 0, 1, 0, 3, 0, 1, 0, 2, 0, 1, 0, \circled{4}, 0, 1, 0, 2, 0, 1, 0, 3, 0, 1, 0, 2, 0, 1, 0, \circled{6}, 0, 1, 0, 2, 0, 1, 0, 3, 0, 1, 0, 2, 0, 1, 0, \circled{4}, 0, 1, 0, 2, 0, 1, 0, \ldots
\end{tblr}
\caption{The terms of \svpn{2} (terms differing from \seqnum{A346070} circled).}
\end{table}

\noindent The differences, moreover, are relatively easy to explain. Sequence \seqnum{A346070}'s initial terms match \svpn{2} mod 4, and four 90-degree turns in a direction are the same as none, 5 are the same as 1, 6 are the same as 2, and so on. So, a sequence interpreted in terms of of 90-degree turns ought to produce the same figure when read mod 4 as when read unmodified. 

Upon closer examination, however, something is off. The first few terms of \seqnum{A346070} produce the following figure:

\begin{figure}[H]
\centering
\includegraphics{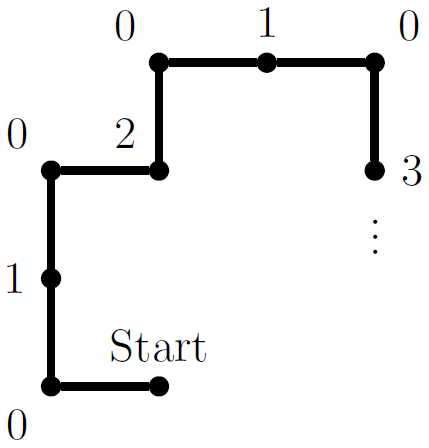}
\caption{A curve with 90-degree turns representing \seqnum{A346070}.}
\end{figure}

\noindent In contrast, an analogous segment of the curve generated by \svpn{2} looks like this:

\begin{figure}[H]
\centering
\includegraphics{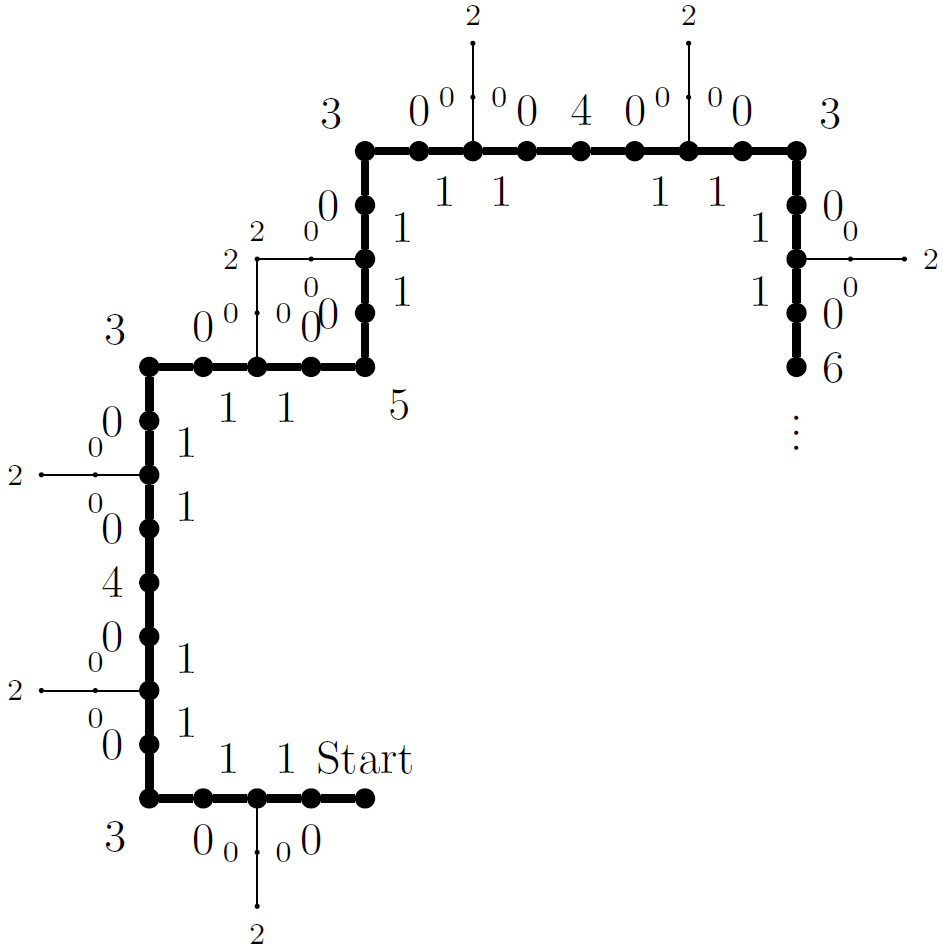}
\caption{A curve with 90-degree turns representing \svpn{2}.}
\end{figure}

\noindent So, even though the two sequences appear to be equivalent, and it would seem that both \textit{ultimately} generate the \levy Dragon, they construct that curve from different building blocks. What is going on?

The low-level differences are due to different mappings between terms and turns. The entry for \seqnum{A346070} suggests we map 0 to a right turn, 1 to no turn, 2 to a left turn, and 3 to an about-face. The mapping we suggested above, in contrast, treats the terms in \svpn{2} as the number of 90-degree turns to make in a single direction.

With such low-level differences, however, how can we be sure that both sequences produce the same geometrical figure? More specifically, can be confident that \svpn{2} produces the \levy Dragon?

\subsection{Motivation}

Before we continue, we ought to address a potential concern: why go to the effort of proving that \svpn{2} produces the \levy Dragon (using the term-to-turn mapping we have proposed), when the entry for \seqnum{A346070} has all but confirmed that it does (using a different mapping)?

The answer is partly philosophical and partly pedagogical. The mapping proposed by the entry for \seqnum{A346070} can seem arbitrary, at least at first. In it, numerals appear to categorize turns but not \textit{quantify} them. So, why not use \{1, 2, 3, 4\}, or \{-2, -1, 1, 2\}, instead of \{0, 1, 2, 3\}? In fact, why use numerals rather than letters, or perhaps arrows? 

A central goal of this paper is to help students see \emd and to help instructors help students see \emd a beautiful structure in something important but (apparently) messy. To that end, it would be helpful to minimize (apparent) arbitrariness where we can. We will have to choose \textit{some} interpretation if we want to get geometry out of integer sequences, but the ``counting turns'' interpretation we employ here is the simpler and less \textit{ad hoc} option. (And, perhaps equally importantly, it creates an interesting challenge.)

\subsection{A sequence of spikes and turns}
Our primary problem comes from the fact that the figure generated by our interpretation of \svpn{2} is ``hairy'' or ``spiky.'' Those spikes, furthermore, seem to be generated by the frequent recurrence of the subsequence, $\langle 0, 1, 0, 2, 0, 1, 0\rangle$, within \svpn{2}.

\begin{figure}[H]
\centering
\includegraphics{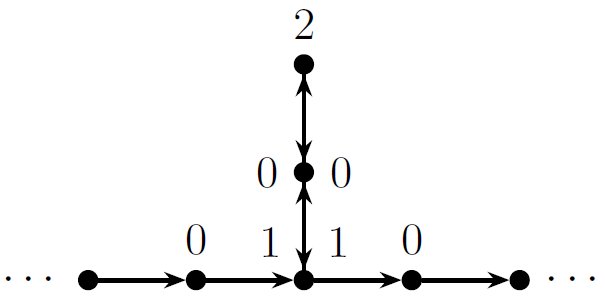}
\caption{A ``spike'' on the \levy Dragon.}
\end{figure}

\noindent In fact, it might be more accurate to say that \svpn{2} consists of instances of that subsequence connected by a series of other numbers, each greater than 2.

\begin{table}[H]
\centering
\begin{tblr}{colspec={Q[l]}}
{\color{gray}0, 1, 0, 2, 0, 1, 0,} \textbf{3}, {\color{gray}0, 1, 0, 2, 0, 1, 0,} \textbf{4}, {\color{gray}0, 1, 0, 2, 0, 1, 0,} \textbf{3}, {\color{gray}0, 1, 0, 2, 0, 1, 0,} \textbf{5},\\ 
{\color{gray}0, 1, 0, 2, 0, 1, 0,} \textbf{3}, {\color{gray}0, 1, 0, 2, 0, 1, 0,} \textbf{4}, {\color{gray}0, 1, 0, 2, 0, 1, 0,} \textbf{3}, {\color{gray}0, 1, 0, 2, 0, 1, 0,} \textbf{6},\\ 
{\color{gray}0, 1, 0, 2, 0, 1, 0,} \textbf{3}, {\color{gray}0, 1, 0, 2, 0, 1, 0,} \textbf{4}, {\color{gray}0, 1, 0, 2, 0, 1, 0,} \textbf{3}, {\color{gray}0, 1, 0, 2, 0, 1, 0,} \textbf{5},\\ 
{\color{gray}0, 1, 0, 2, 0, 1, 0,} \textbf{3}, {\color{gray}0, 1, 0, 2, 0, 1, 0,} \textbf{4}, {\color{gray}0, 1, 0, 2, 0, 1, 0,} \textbf{3}, {\color{gray}0, 1, 0, 2, 0, 1, 0,} \textbf{7},\\. . .
\end{tblr}
\caption{The terms between the spikes in \vpn{2}.}
\end{table}

From comparing figures 4 and 5, furthermore, it would seem that the ``T'' shapes play the same role in the \svpn{2} version of the \levy Dragon as single line segments do in the \seqnum{A346070} version of thereof. Though you take a small detour along the way, the overall effect of traversing a T is the same as traversing a single straight line.

If we want to show that our interpretation of \svpn{2} produces the \levy Dragon, then, we need to show that the series of ``other numbers'' between copies of the subsequence  $\langle 0, 1, 0, 2, 0, 1, 0\rangle$ match the series of turns made by the \levy Dragon as normally defined.

Since  $\langle 0, 1, 0, 2, 0, 1, 0\rangle$ is seven elements long, the numbers that connect each instance of that subsequence to the next will be found at indexes that are multiples of 8 (i.e., of $2^3$). This means they will all be at indexes of the form, $i=f\times2^3$, where $f$ is a positive integer. What we want to do, then, is to figure out what algorithm generates the sequence of turns in the standard \levy Dragon, and then see whether that algorithm generates the same sequence of numbers as are found in \svpn{2} at indexes of the form, $i=f\times2^3$.

\subsection{Generating the \levy Dragon}
The \levy Dragon is generated by starting with an isosceles right triangle, then constructing smaller isosceles right triangles on the equal sides of that triangle, then constructing smaller right isosceles triangles on their sides, and so on (\levy \cite[p.\ 207]{levy}; Riddle \cite{riddle}).

\begin{figure}[H]
\centering
\includegraphics[scale=0.8]{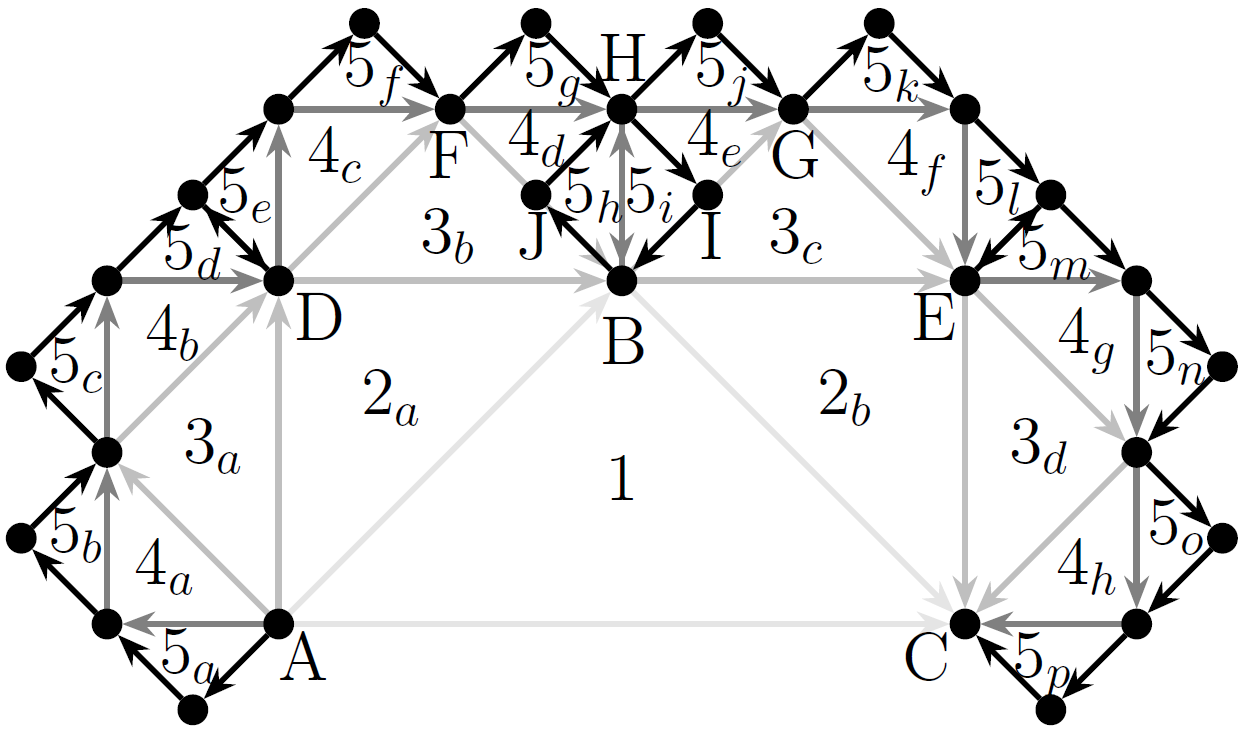}
\caption{The first five iterations of the \levy Dragon.}
\end{figure}
 
Alster \cite[p.\ 857]{alster} points out that once a vertex has been introduced into the figure in one iteration, it will be retained in every iteration thereafter. Alster \cite[pp.\ 857--858]{alster} then focuses on the relative orientations of the line segments that meet at each vertex, and how those orientations change from iteration to iteration. Consider the sequence of edge pairs that meet at B in Figure 7, and the sequence of turns a turtle crawling along the curve from A to C would make at B.\footnote{Note that vertex H, in Figure 8, is actually two vertexes that happen to exactly coincide.}

\begin{table}[H]
\centering
\begin{tblr}{colspec={Q[r]|Q[c]|Q[l]}, rowspec={Q[m]|}}
\textbf{Iteration} & \textbf{Edges} & \textbf{Turn}\\
1 & AB, BC & 90 degrees clockwise\\
2 & DB, BE & 0 degrees\\
3 & FB, BG & 90 degrees counterclockwise\\
4 & H\textsubscript{1}B, BH\textsubscript{2} & 180 degrees\\
5 & IB, BJ & 90 degrees clockwise
\end{tblr}
\caption{The edges that will meet at B, and the turn a turtle would make at B.} 
\end{table}

From B's point of view, the edge leading into it rotates clockwise 45 degrees from iteration to iteration, while the edge leading out of it rotates counterclockwise 45 degrees. This is because B, like every new vertex introduced into the curve in a given iteration, will be the ``apex'' vertex in an isosceles right triangle \emd the vertex where that triangle's two equal sides meet. Those two sides will initially meet at a right angle, one equivalent to a 90-degree clockwise turn if you traverse the curve from A to C. In the next iteration, those two sides will spawn two new isosceles triangles. The one that is closer to A in the curve will be rotated counterclockwise by 45 degrees, relative to the triangle on whose side it spawned, while the one that is closer to C in the curve will be rotated clockwise by 45 degrees. This means the edges that meet at the vertex will both be parallel to the base (hypotenuse) of the triangle that spawned them, and thus will meet at the vertex to form a straight line.

From this point on, the two sides that meet at a given vertex will belong to separate triangles. Each will still be rotated 45 degrees relative to the side on which the triangle to which it belongs spawned, however, and thus each side will be rotated $k\times 45$ degrees relative to the side of the original triangle that introduced the vertex (where $k$ is the number of iterations since the vertex and its triangle were introduced).

In light of this, the angle our turtle must turn at a given vertex in iteration $i$ must differ from the angle it turns in iteration $i+1$ by at most 90 degrees (with 45 of those degrees being due to the rotation of the incoming edge, and the other 45 degrees being due to the rotation of the outgoing edge). If we measure angles clockwise, we will think of our turtle as initially turning 90 degrees at a given vertex, then 0 degrees, then $-90$ degrees, then $-180$, $-270$, $-360$, etc. In contrast, if we measure the angles counterclockwise, we will think of our turtle as initially turning 270 degrees, then 360 degrees, then 450, then 540, etc.

If we think of the clockwise angles as divided into 90-degree turns, we would see our turtle as turning clockwise once at a given vertex, then 0 times, then $-1$ times, then $-2$ times, and so on. It may be difficult to make sense of a negative number of turns, however. We might therefore choose to measure the angles counterclockwise, in which case our turtle would make three 90-degree turns at a vertex, then 4, then 5, etc.

Since the counterclockwise framing avoids talk of negative turns \emd and because it will smooth out the transition back to discussing multiples of 8 in \svpn{2} \emd we will adopt the counterclockwise framing here. Doing so leads to the formula $$t = (270 + j\times90)/90 = 3 + j$$ where $t$ is the number of turns to make at a given vertex and $j$ is the number of iterations since that vertex was introduced.

From that formula, furthermore, we can derive an algorithm for generating the entire sequence of turns. Let each vertex in the figure be represented by the number of turns a turtle crawling along the figure from A to C will make at that vertex during the current iteration. Since A and C mark the beginning and end of the curve, no turns or numbers will be associated with them. 

Each iteration will begin by inserting a new vertex between every pair of vertexes already in the figure. Whenever a new vertex is inserted, its associated number will be 3. A vertex's number will then increase by 1 with each new iteration (as the formula above dictates).

More formally, the algorithm for generating the sequence of turns in the \levy Dragon, framed in terms of 90-degree counterclockwise turns, is as follows.

\begin{enumerate}
\item Start with the sequence $\langle$ 3 $\rangle$.
\item Increase every number in the sequence by 1.
\item Between every number currently in the sequence, insert a 3.
\item Append a 3 to the beginning of the sequence and another to the end.
\item Go to step 2, or stop, as desired.
\end{enumerate}

\noindent Following this algorithm, we would generate this sequence of sequences:

\begin{table}[H]
\centering
\begin{tblr}{colspec={Q[c]|Q[c]|Q[c]}, rowspec={Q[m]|}}
Start & Increment & Insert 3s \\
$\langle$3$\rangle$		&  $\langle$4$\rangle$	& $\langle$3, 4, 3$\rangle$\\
$\langle$3, 4, 3$\rangle$		&  $\langle$4, 5, 4$\rangle$	& $\langle$3, 4, 3, 5, 3, 4, 3$\rangle$	\\
$\langle$3, 4, 3, 5, 3, 4, 3$\rangle$ &  $\langle$4, 5, 4, 6, 4, 5, 4	$\rangle$	& $\langle$3, 4, 3, 5, 3, 4, 3, 6, 3, 4, 3, 5, 3, 4, 3$\rangle$	\\
\end{tblr}
\caption{Three iterations in generating the sequence of turns in the \levy Dragon.}
\end{table}

\noindent This sequence \textit{looks} like the subsequence we were discussing above: the terms in \svpn{2} at indexes that are multiples of 8. But is it?

\begin{theorem}
The algorithm for generating the \levy Dragon's turns also generates the terms in \svpn{2} at indexes that are multiples of 8.
\end{theorem}

\begin{proof}
In what follows, let \lseq{} be the integer sequence generated by the algorithm above, and let \lseq{i} be the $i$th term in that sequence.

Every new term introduced into \lseq{} will be a 3, and those 3s will be introduced at odd indexes (1, 3, 5, 7, etc.). This matches the terms in \svpn{2} at indexes that are odd multiples of 8. Each such index has the form $i = o \times 2^3$, where $o$ is an odd positive integer. Now, the term at index $i$ in \svpn{2} represents the largest integer $k$ such that $2^k$ divides $i$. But the largest $k$ such that $2^k$ divides $o \times 2^3$ is just 3, since $o$ is odd and odd numbers are not divisible by 2.

Therefore, the algorithm above places 3s at all odd indexes in \lseq, which matches the 3s we will find at the indexes in \svpn{2} that are odd multiples of 8. Next, then, we need to check whether the terms at even indexes in \lseq{} also match the terms in \svpn{2} at indexes that are even multiples of 8.

The terms at even indexes in \lseq{} all start as 3s at odd indexes, and then are ``shifted back'' and increased by 1 with each iteration. Consider the development from iteration to iteration of the term just introduced at index $o$. The term, \lseq{o}, will be a 3, as we just discussed. In the next iteration, however, it is increased by 1 and $o$ new terms are inserted into the sequence before it. Thus, the term that started at index $o$, as a 3, ends up at index $2o$, as a 4. Is this the same number we will find in \svpn{2} at index $i = 2o \times 2^3$?

The number at index $i = 2o \times 2^3$ in \svpn{2} is the largest $k$ such that $2^k$ divides $2o\times2^3=2^1o\times2^3=o\times2^{3+1}=o\times2^4$. But $o$ is odd, and thus the largest $k$ such that $2^k$ divides $o\times2^4$ must be 4. So, both the term at index $2o$ in \lseq{} and the term at index $2o \times 2^3$ in \svpn{2} will be 4.

In the next iteration, the term at index $2o$ in \lseq{} is increased by 1 and $2o$ new terms are inserted into \lseq{} before it. Thus, the term that started at index $o$, as a 3, ends up at index $4o$, as a 5. Will we also find a 5 at index $i = 4o\times 2^3$ in \svpn{2}? 

The number at index $i = 4o\times 2^3$ in \svpn{2} is the largest $k$ such that $2^k$ divides $4o\times2^3=2^2o\times2^3=o\times2^{3+2}=o\times2^5$. But $o$ is odd, and thus the largest $k$ such that $2^k$ divides $o\times2^5$ must be 5. So, both the term at index $4o$ in \lseq{} and the term at index $4o \times 2^3$ in \svpn{2} will be 5.

This pattern of matching continues, furthermore. After $j$ iterations, the term introduced at index $o$ will be at index $2^jo$ and will have a value of $3+j$. This, furthermore, is the value that will be at index $i=2^jo\times 2^3$ in \svpn{2}, since the largest $k$ such that $2^k$ divides $2^jo\times2^3=o\times2^{3+j}$ is $k=3+j$. 

We conclude, then, that all new terms introduced in a given iteration are placed at the correct indexes (to match the terms in \svpn{2} at indexes that are odd multiples of 8) and given the correct values, while all the terms already present have their values adjusted and are moved to indexes that are correct (to match the terms in \svpn{2} at indexes that are even multiples of 8) for those new values. This leaves no gaps, furthermore. Old terms do not jump arbitrarily to new indexes. New terms are inserted into the sequence, pushing the terms that are already there, back. An old term only ends up at a given index because all the previous indexes are full.

Therefore, \lseq{i} \emd whether $i$ is odd or even \emd is the largest integer $k$ such that $2^k$ divides $i\times8$. And that means the algorithm for generating the \levy Dragon's turns also generates the terms in \svpn{2} at indexes that are multiples of 8.
\end{proof}

\subsection{Discussion} The ``T'' shapes in the figure generated by \svpn{2} are connected by the same series of turns as the line segments in the standard construction of the the \levy Dragon.  What is more, the portion of each T shape that deviates from a straight line does not grow in size from iteration to iteration. This means that the size of these deviations, relative to the overall figure, becomes negligible as the number of iterations increases. As the number of iterations increases, in other words, the curve generated by \svpn{2} \emd under the interpretation we proposed, above \emd approaches the complete \levy Dragon. 

\section{Other primes and fractals}

We now know that which primes appear in the prime factorizations of which positive integers, and what their exponents are therein, is determined by an infinite set of fractal integer sequences. We also know that one of those sequences, \svpn{2}, produces the \levy Dragon fractal when represented by 90-degree turns in the number line. What this shows, I believe, is that our extension of the concept, ``fractal,'' from geometry to integer sequences was not a stretch. A common structure underlies both.

There is nothing particularly unique about \svpn{2}, furthermore. Here is \svpn{3} represented by 90-degree turns:

\begin{figure}[H]
\centering
\includegraphics{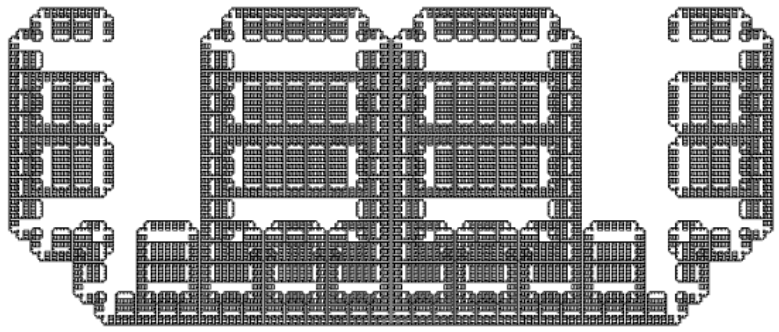}
\caption{The number line with 90-degree turns representing \svpn{3}.}
\end{figure}

\noindent Here is the number line when it represents \svpn{5} by 120 degree turns:

\begin{figure}[H]
\centering
\includegraphics{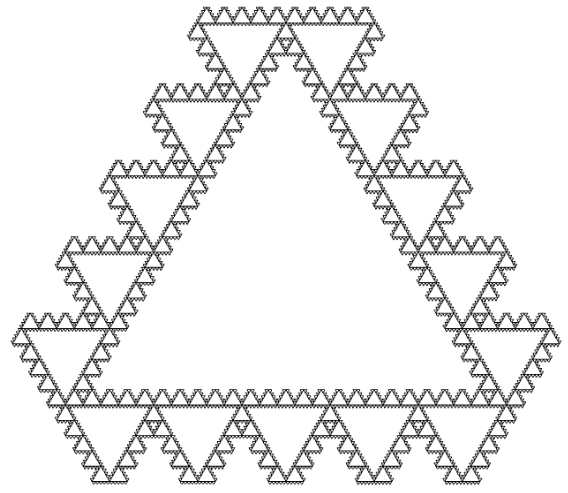}
\caption{The number line with 120-degree turns representing \svpn{5}.}
\end{figure}
\noindent And here is \svpn{7} represented by 135 degree turns:

\begin{figure}[H]
\centering
\includegraphics{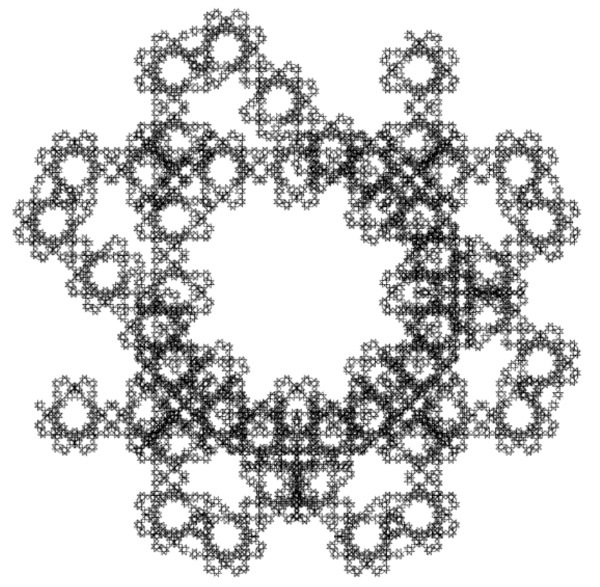}
\caption{The number line with 135-degree turns representing \svpn{7}.}
\end{figure}
 
These angles are not unique, furthermore. I have also foundfractal figures at 120 degrees for \svpn{2} and at 60 degrees for \svpn{5}, for example.

\begin{figure}[H]
\centering
\includegraphics{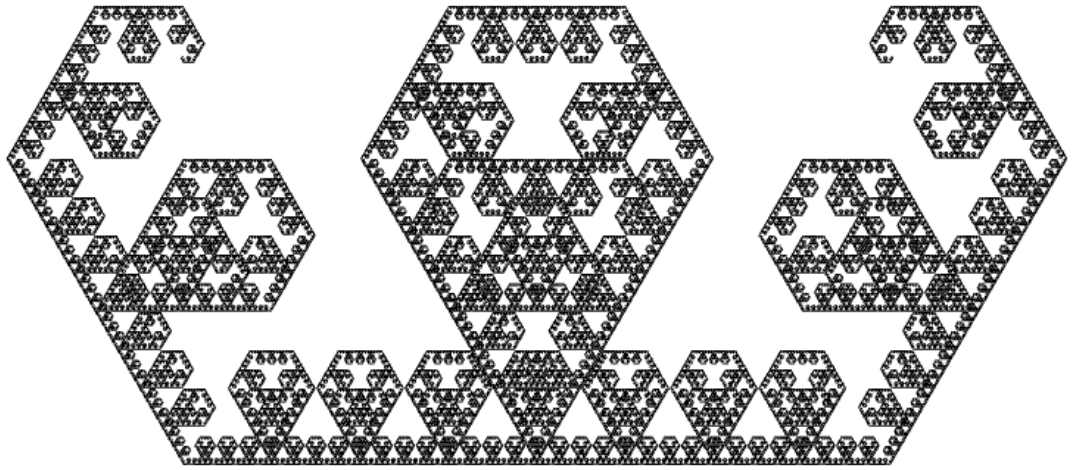}
\caption{The number line with 120-degree turns representing \svpn{2}.}
\end{figure}

\begin{figure}[H]
\centering
\includegraphics{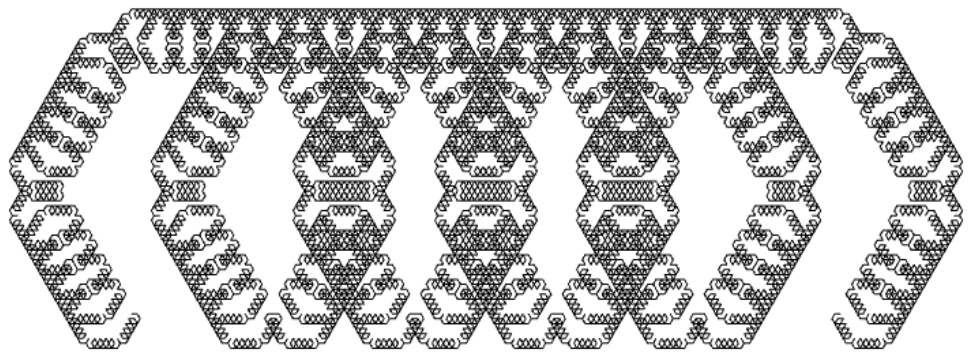}
\caption{The number line with 60-degree turns representing \svpn{5}.}
\end{figure}

This, of course, is not a formal argument that the term ``fractal'' refers to objects with the same underlying structure in both geometry and integer sequences. I have not made that claim precise enough for it to be given a formal demonstration. I simply hope the evidence of a connection between fractal integer sequences and fractal figures is strong enough that others might be encouraged to investigate the topic further. 

Some of the questions for which I would like answers are:
	
\begin{enumerate}
\item Why do some angles produce fractal figures for a given sequence in \svpns{p}, while others do not? (E.g., 90 and 120 work for \svpn{2}, but 60 and 30 don't.)

\begin{table}[H]
\centering
\begin{tblr}{colspec={Q[c]Q[c]}, rowspec={Q[m]}}
\includegraphics{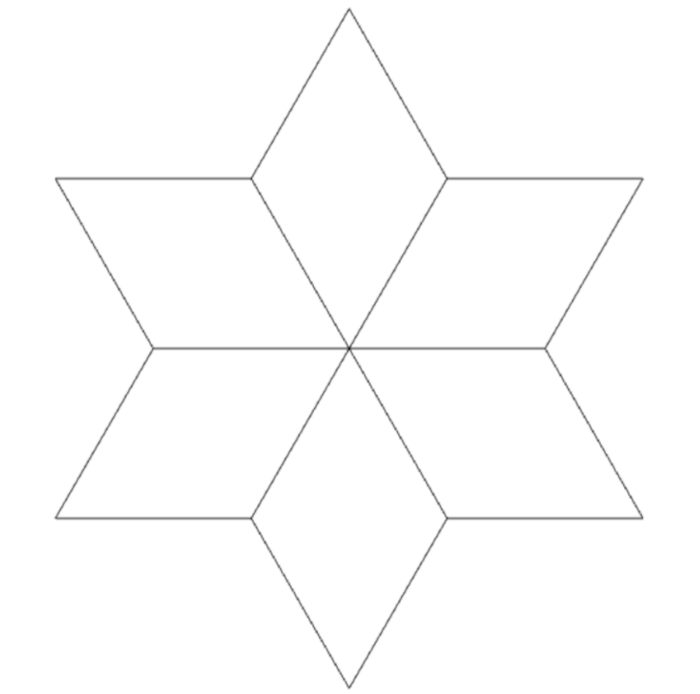} & \includegraphics{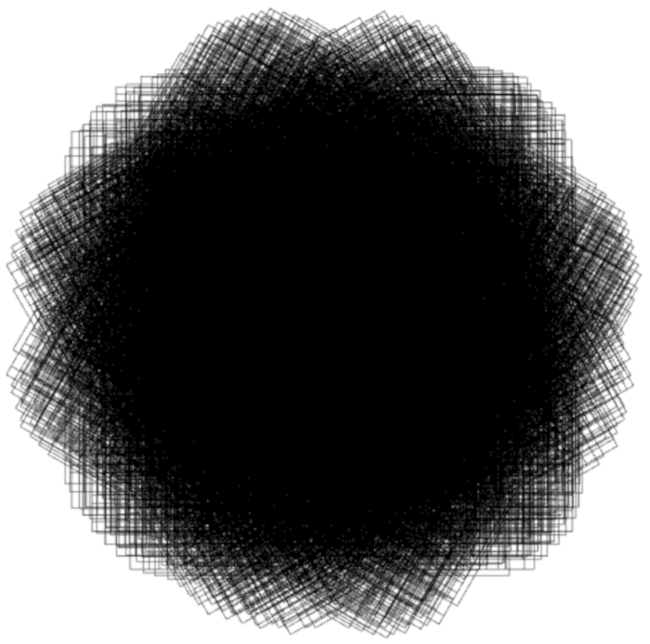}\\
\svpn{2} with 60-degree turns & \svpn{2} with 30-degree turns
\end{tblr}
\caption{Some angles produce non-fractal figures.}
\end{table} 

\item When a particular choice of angle fails to produce a fractal, why does it sometimes produce a clear geometrical pattern, and other times produce a mess?

\item Why do different angles produce fractal figures for different \svpn{p} sequences? (For example, 120 degrees works for \svpn{2} and \svpn{5}, but not \svpn{3}.)

\item Why does the number of iterations a given sequence has gone through in its development change the fractal figure produced in some cases? On which shape should we say the figure generated by the sequence converges, if any?

The figure generated by \svpn{2} at 90 degrees always approximates the \levy Dragon, for example, though the figure rotates depending on the number iterations the sequence has been put through. The figure generated by \svpn{2} at 120 degrees, however, changes in a regular fashion from one iteration of the sequence to the next. Why?

\begin{table}[H]
\centering
\begin{tblr}{colspec={Q[c]Q[c]}, rowspec={Q[m]}}
\includegraphics{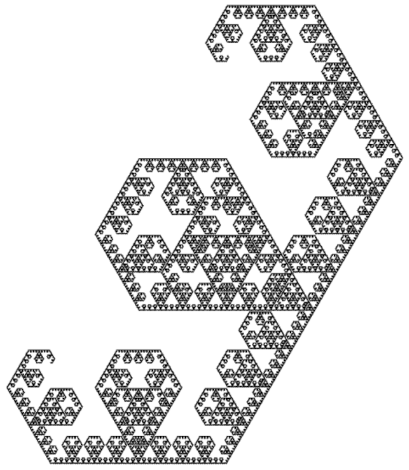} & \includegraphics{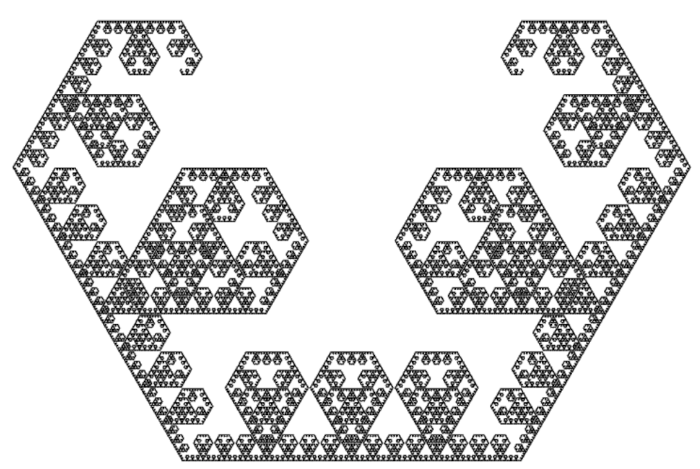}\\
\svpn{2}, 17th iteration, 120 degrees & \svpn{2}, 18th iteration, 120 degrees
\end{tblr}
\caption{The figure produced sometimes changes as iterations increase.}
\end{table} 

\item Is there at least one angle for each \svpn{p} that would generate a fractal figure, or is there some $q$ such that no \svpn{p} will produce a fractal figure if $p>q$? 

\item If the fact that fractal integer sequences ``generate'' fractal figures is evidence (as I claim) that the two classes of object share an underlying structure, what does the fact that those sequences also generate \textit{non}-fractal figures tell us? Can fractal and non-fractal figures share the same structure in some (interesting) sense?

\item Can multiple sequences from \svpns{p} be combined and yet still generate a fractal figure? For example, if $c(n) = a\times\text{\vpn{3}} + b\times\text{\vpn{5}}$, could we find an angle such that the sequence generated by $c(n)$ would produce a fractal figure?

\item The process above, in which we generate integer sequences and interpret them to inform the drawing of geometrical figures, is similar to that involved in Lindenmayer systems. (Much the same phenomena occur with figures generated by Lindenmayer systems as are listed in points (1), (2), (3), and (4), above, furthermore.) Can the duplicate-concatenate-increment algorithm for generating \svpn{p} be translated into a string-replacement algorithm like that employed by Lindenmayer systems? Can the quantitative mapping from terms-to-turns employed above be translated into the sort of categorical mapping usually employed by Lindenmayer systems?
\end{enumerate}

\section{Appendix: The Heighway Dragon and the odd part of $n$}
In this final section, we will need the concepts of the ``odd'' and ``even parts of $n$.'' The even part of $n$ is the largest power of 2 that divides $n$ and can be found in OEIS sequence \seqnum{A006519}. That is, the even part of $n$ is $2^{v_2(n)}$. This is the ``part'' of $n$ associated with the L\'{e}vy Dragon, as explained above.

The odd part of $n$, in contrast, is the largest odd factor of $n$, and can be found in OEIS sequence \seqnum{A000265}. That is, the odd part of $n$ is $n$ divided by its even part, or the product of all the multiplicands with odd bases in the prime factorization of $n$. (For prime factorizations that contain no multiplicands with odd bases, the odd part of $n$ is just 1.) 

In what follows, we will refer to the odd part of some particular $n$ as $o(n)$, and the entire odd part of $n$ sequence as \oddpart. It is our goal in this section to see how \oddpart ~is associated with the Heighway Dragon. We will first determine how an integer sequence representing the turns in that curve can be generated. We will then determine how \oddpart ~can be generated. And finally, we will show that the algorithm for generating \oddpart ~is equivalent to the algorithm for generating the sequence of turns in the Heighway Dragon.

\subsection{An algorithm for generating the turns in the Heighway Dragon}
John Heighway and William Harter discovered what we now typically refer to as \textit{the} Dragon Curve while exploring the figures produced by repeatedly folding a strip of paper in half (always in the same direction), then opening each fold in the paper to 90 degrees and viewing the whole thing edge-on (see Knuth and Davis 2011 \cite{knuth}).

In the interest of mathematical purity, we will represent the strip of paper with a horizontal line connecting two vertexes: A and B. We will, furthermore, imagine a mathematical ant making repeated journeys along the figure from A to B. 

\begin{figure}[H]
\centering
\includegraphics{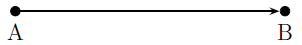}
\caption{A ``strip of paper'' to be folded into the Heighway Dragon}
\end{figure}

\noindent On its first journey, our ant will travel from left to right, starting at A and ending at B. Before its second journey, however, we will ``fold'' AB in half. That is, we will insert a vertex C into AB at its midpoint. 

\begin{figure}[H]
\centering
\includegraphics{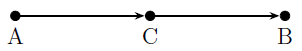}
\caption{Beginning of the first fold}
\end{figure}

\noindent We will then rotate CB counterclockwise around C, till B and A coincide.

\begin{figure}[H]
\centering
\includegraphics{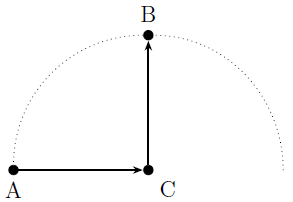}
\caption{First fold, at 90 degrees}
\end{figure}

\noindent After repeated foldings, we will open each fold to 90 degrees. If we then send the ant on another journey from A to B, it would turn 90 degrees left (counterclockwise) at C.

It is not yet time to unfold, however. So far, we have made one complete fold, and our ant would travel AC from left to right, then CB from right to left.

\begin{figure}[H]
\centering
\includegraphics{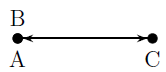}
\caption{First fold, completed}
\end{figure}

Next, we make our second fold. To start, we insert two new vertexes, D\textsubscript{1} and D\textsubscript{2}, at the midpoints of AC and CB respectively. (In the figures below, we will represent the ``paper'' as partially unfolded, so that we can see which vertex and edge is which.) 

\begin{figure}[H]
\centering
\includegraphics{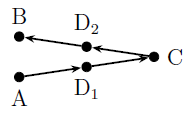}
\caption{Beginning of second fold (with non-overlapping vertexes for clarity)}
\end{figure}

\noindent Then, to complete the fold, we will rotate CD\textsubscript{2} counterclockwise around D\textsubscript{2}, and D\textsubscript{1}C counterclockwise around D\textsubscript{1}, till C coincides with A and B. 

\begin{figure}[H]
\centering
\includegraphics{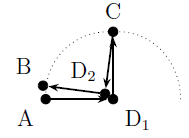}
\caption{Second fold, partially complete (with non-overlapping vertexes for clarity)}
\end{figure}

We should note that D\textsubscript{1} was inserted into an edge (AC) our ant would have traversed left to right \emd if the figure were in its fully-folded state \emd and D\textsubscript{2} was inserted into an edge (CB) that our ant would have traversed right to left. When we open all the folds to 90 degrees, our ant will turn 90 degrees to the left at D\textsubscript{1}, and 90 degrees to the right at D\textsubscript{2}. Nothing about its turn at C would change, however; it would still make a left turn there as well.

\begin{figure}[H]
\centering
\includegraphics{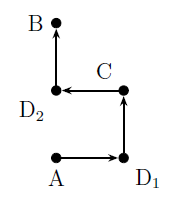}
\caption{After two folds, opened to 90 degrees}
\end{figure}

We learn a number of things from just these first two folds.

\begin{enumerate}
\item A vertex inserted into an edge that the ant would traverse from left to right (were the figure in its fully-folded state) becomes a left turn when all folds are opened to 90 degrees. Otherwise, it becomes a right turn.
\item Once a vertex has been inserted and a fold made at that vertex, future foldings may change the place of that vertex in the overall sequence of vertexes, but do not change the direction our ant will have to turn when it gets to that vertex. 
\item The direction our ant has to travel on its journey from A to B always reverses with each new edge (when the figure is in its fully-folded state).
\end{enumerate}

From these facts we can derive an algorithm for generating the sequence of turns in Heighway Dragon. We start with a sequence consisting of two 0s.

\begin{enumerate}
\item Between each pair of terms in the sequence, insert a positive integer. Specifically, if the first term in the pair was at an odd index at the start of this iteration, insert a 1. Otherwise, insert a 3.
\item Go to step 1, or step 3, as desired.
\item Remove the 0s at the start and end of the sequence. Then stop.
\end{enumerate}

For example:

\begin{table}[H]
\centering
\begin{tblr}{colspec={Q[l]}}
0, 0\\
{\color{gray}0,} 1, {\color{gray}0}\\
{\color{gray}0,} 1, {\color{gray}1,} 3, {\color{gray}0}\\
{\color{gray}0,} 1, {\color{gray}1,} 3, {\color{gray}1,} 1, {\color{gray}3,} 3, {\color{gray}0}\\
{\color{gray}0,} 1, {\color{gray}1,} 3, {\color{gray}1,} 1, {\color{gray}3,} 3, {\color{gray}1,} 1, {\color{gray}1,} 3, {\color{gray}3,} 1, {\color{gray}3,} 3, {\color{gray}0}\\
\end{tblr}
\caption{Four iterations in generating the sequence of turns in the Heighway Dragon}
\end{table}

So, after four iterations, our sequence is $\langle 1, 1, 3, 1, 1, 3, 3, 1, 1, 1, 3, 3, 1, 3, 3 \rangle$. 

\begin{figure}[H]
\centering
\includegraphics{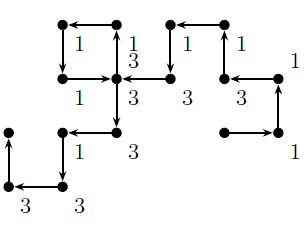}
\caption{The Heighway Dragon after four iterations}
\end{figure}

If we run a search for this sequence in the OEIS, we get sequence \seqnum{A099545}, entitled, ``Odd part of n, modulo 4.''

\begin{table}[H]
\centering
\begin{tblr}{colspec={Q[l]}}
1, 1, 3, 1, 1, 3, 3, 1, 1, 1, 3, 3, 1, 3, 3, 1, 1, 1, 3, 1, 1, 3, 3, 3, 1, 1, 3,\ldots
\end{tblr}
\caption{OEIS sequence A099545, the odd part of $n$, mod 4}
\end{table}

\noindent The ``odd part of $n$'' (\textit{not} mod $4$) can be found in sequence \seqnum{A000265}:

\begin{table}[H]
\centering
\begin{tblr}{colspec={Q[l]}}
1, 1, 3, 1, 5, 3, 7, 1, 9, 5, 11, 3, 13, 7, 15, 1, 17, 9, 19, 5, 21, 11, 23, 3, 25,\\ 
13, 27, 7, 29, 15, 31, 1, 33, 17, 35, 9, 37, 19, 39, 5, 41, 21, 43, 11, 45, \ldots
\end{tblr}
\caption{OEIS sequence A000265, the odd part of $n$}
\end{table}

\subsection{The fractal nature of \oddpart}

Whenever $n$ is odd, $2^{\text{\vpn{2}}}=1$, since $\text{\vpn{2}}=0$ and $2^0=1$. Now, the odd part of $n$ \emd i.e., $o(n)$ \emd is obtained by dividing $n$ by $2^{\text{\vpn{2}}}$, and when $n$ is odd, $2^{\text{\vpn{2}}}=1$. So, to obtain $o(n)$ when $n$ is odd, we just divide $n$ by $1$. But, this means that $o(n)$ just \textit{is} $n$ when $n$ is odd. This leads us to expect the numbers at odd indexes in \seqnum{A000265} to just be the odd positive integers, and this is what we find. 

Once we see this, we might then notice that the sequence is fractal, as Kerry Mitchell \cite{mitchell} points out in the OEIS entry for \seqnum{A000265}. While the terms at odd indexes just are the odd integers (in order, starting with 1), the terms at even indexes are the sequence itself. 

As usual, however, we would like a demonstration that the sequence is fractal, if such a demonstration can be obtained. To start, we will figure out how the decimation rule for \oddpart ~works, and obtain a formula for determining which terms at which indexes are to be ``removed,'' and which are to be ``retained.'' Next, we will then see how we can reverse the decimation process to construct an integer sequence. And third, we will show that the sequence so constructed is identical to \oddpart.

\subsubsection{How the decimation rule works}
Mitchell proposes that \oddpart ~is fractal, and that its decimation rule is to ignore its odd numbered elements, retaining only its even numbered elements. If this is correct, then we ought to be able to apply that same decimation rule to the sub-sequence we located via its first application, and then apply it again to sub-sequence we located via its second application, etc. Each time, furthermore, we should be ``removing'' (i.e., ignoring) the odd positive integers, in order, starting with 1. 

Consider, then, the indexes of the terms we will be ignoring with each application of the decimation rule. With the first application, it will be the terms at odd indexes \emd the terms at indexes of the form $o\times2^0$, where $o$ is an odd positive integer. That is, we will be ignoring the terms at indexes $1\times2^0=1, 3\times2^0=3, 5\times3^0=5$, etc. This will leave the terms at even indexes, which are the terms at indexes of the form $k\times2^1$, where $k$ is any positive integer. That is, we will be retaining the terms at indexes $1\times2^1=2, 2\times2^1=4, 3\times2^1=6$, etc.

Next, we will ignore every other term in the sub-sequence found at indexes of the form $k\times2^1$. That is, we will be ``removing'' the terms at odd multiples of $2^1$, and thus which have the form $o\times2^1$ (where $o$ is some odd positive integer). That is, we will be removing the entries in our original list at indexes $1\times2^1=2, 3\times2^1=6, 5\times2^1=10$, etc. This will leave only the even multiples of $2^1$, which will all have the form $2k\times2^1=k\times2^2$. That is, they will be the entries at indexes $1\times2^2=4, 2\times2^2=8, 3\times2^2=12$, etc.

In general, the $j$th sub-sequence we find in \seqnum{A000265} will be obtained by ignoring all terms at indexes of the form $o\times2^j$ (where $o$ is an odd positive integer) and focusing on all the terms at indexes of the form $k\times2^{j+1}$ (where $k$ is any positive integer) \emd assuming that $j$ starts at $0$. 

\subsubsection{How to reconstruct the sequence}
We can now use what we have just learned to figure out how to construct the sequence in which the entries at odd indexes are the odd positive integers (in order, starting with 1) and in which the terms at even indexes are just the sequence itself. Since we are supposed to be removing the odd integers with each application of the decimation rule, we ought to be able to iteratively reconstruct the entire sequence by placing the odd integers at indexes of the form $o\times2^j$, where $j$ is the number of the iteration (and starts at 0). That is, at the index $o\times2^j$, we ought to place the term $o$ (e.g., at index $1\times2^3$ we ought to place a $1$, at index $5\times2^4$ we ought to place a $5$, at index $13\times2^{10}$ we ought to place a $13$, and so on).

\begin{table}[H]
\centering
\begin{tblr}{colspec={Q[c]}}
Indexes:     & 1 & 2 & 3 & 4 & 5 & 6 & 7 & 8 & 9 & 10 & 11 & 12 & 13 & 14 & 15\\
Iteration 0: & \textbf{1} &   & \textbf{3} &   & \textbf{5} &   & \textbf{7} &   & \textbf{9} &    & \textbf{11} &     & \textbf{13} &    & \textbf{15}\\
Iteration 1: & 1 & \textbf{1} & 3 &   & 5 & \textbf{3} & 7 &   & 9 & \textbf{5}  & 11 &     & 13 & \textbf{7}  & 15\\
Iteration 2: & 1 & 1 & 3 & \textbf{1} & 5 & 3 & 7 &   & 9 & 5  & 11 & \textbf{3}  & 13 & 7  & 15\\
Iteration 3: & 1 & 1 & 3 & 1 & 5 & 3 & 7 & \textbf{1} & 9 & 5  & 11 & 3  & 13 & 7  & 15\\
\end{tblr}
\caption{Proposed construction of \oddpart}
\end{table}

This certainly looks like the \oddpart ~sequence. But is it? If we iteratively place the odd integers at indexes in a sequence that have the form $o\times2^j$ (where $o$ is an odd integer and $j$ is the number of the iteration, starting with 0), will we obtain a sequence in which the term at index $n$ is the largest odd divisor of $n$?

In the algorithm, we place each $o$ at every index $o\times2^j$. Now, the number at index $n$ is supposed to be the odd part of the integer $n$. So, is $o$ the odd part of $o\times2^j$?

To answer that question, we need to divide $o\times2^j$ by the largest power of 2 by which it is divisible. But $o$ is odd, and thus is not divisible by $2$. This means that $2^j$ just \textit{is} the largest power of 2 that evenly divides $n=o\times2^j$. So, when we divide $o\times2^j$ by $2^j$, we are left with the $o$. This means that $o$ is the odd part of $n$ for every $n$ at an index of the form $o\times2^j$. And that means that the term at index $o\times2^j$ ought to be $o$. 

Therefore, the sequence constructed by reversing the decimation rule above \textit{is} \oddpart. It follows from that, furthermore, that \oddpart is fractal.

But what does all this have to do with the Heighway Dragon?

\subsection{Connecting \oddpart ~to the Heighway Dragon}
Given what we have seen above, we suspect that the sequence of turns in the Heighway Dragon, when represented as numbers of counterclockwise 90-degree turns, is identical to \oddpart ~mod $4$. Furthermore, we have just demonstrated that \oddpart ~(\textit{not} mod $4$) is fractal. What we need to show now is that the sequence of turns in the Heighway Dragon (represented as numbers of counterclockwise 90-degree turns) actually \textit{is} \oddpart ~mod $4$.

We will demonstrate this by showing that the algorithm for generating \oddpart ~(when its outputs are taken mod $4$), mirrors the algorithm for generating the Heighway Dragon. All that is required is that we run the \oddpart ~algorithm ``backward." What would happen if we start with the $j$th iteration of the algorithm for generating \oddpart, then, and worked iteration by iteration back down to 1?

Starting with $j=3$, we would have:

\begin{table}[H]
\centering
\begin{tblr}{colspec={Q[c]}}
Indexes:     & 1 & 2 & 3 & 4 & 5 & 6 & 7 & 8 & 9 & 10 & 11 & 12 & 13 & 14 & 15\\
Iteration 3: &  &  &  &  &  &  &  & \textbf{1} &  &  &  &  &  &   & \\
Iteration 2: &  &  &  & \textbf{1} &  &  &  & 1 & &  &  & \textbf{3} & & & \\
Iteration 1: &  & \textbf{1} &  & 1 &  & \textbf{3} &  & 1  &  & \textbf{5}  &  &  3  &  & \textbf{7}  & \\
Iteration 0: & \textbf{1} &  1 & \textbf{3} & 1 & \textbf{5} &  3 & \textbf{7} & 1  & \textbf{9} &  5  & \textbf{11} & 3  & \textbf{13} &  7  & \textbf{15}\\
\end{tblr}
\caption{Construction of \oddpart, reversed}
\end{table}

Running the algorithm for generating the odd part of $n$ with $j$ decreasing instead of increasing looks suspiciously like the construction of the sequence of turns for the Heighway Dragon,  if we align each introduced turn with its position after some finite number of iterations:

\begin{table}[H]
\centering
\begin{tblr}{colspec={Q[c]}}
Index: & 0 & 1 & 2 & 3 & 4 & 5 & 6 & 7 & 8 & 9 & 10 & 11 & 12 & 13 & 14 & 15 & 16\\
Start: & 0 &&&&&&&&&&&&&&&& 0\\
Iteration 0: & 0 &&&&&&&& \textbf{1} &&&&&&&& 0\\
Iteration 1: & 0 &&&& \textbf{1} &&&& 1 &&&& \textbf{3} &&&& 0\\
Iteration 2: & 0 && \textbf{1} && 1 && \textbf{3} && 1 && \textbf{1} && 3 && \textbf{3} && 0\\
Iteration 3: & 0 & \textbf{1} & 1 & \textbf{3} & 1 & \textbf{1} & 3 & \textbf{3} & 1 & \textbf{1} & 1 & \textbf{3} & 3 & \textbf{1} & 3 & \textbf{3} & 0\\
\end{tblr}
\caption{Four iterations in generating the sequence of turns in the Heighway Dragon}
\end{table}

For each $1$ or $3$ in the odd part of $n$ table, there is a corresponding $1$ or $3$ in the Heighway Dragon table. If we take the numbers larger than $3$ (in the odd part of $n$ table) mod $4$, we find the same number in the same position in the Heighway Dragon table.

Once again, however, we should look for a better demonstration than is provided by a visual scan. Each iteration of the algorithm for generating the Heighway Dragon's turns amounts to inserting alternating $1$s and $3$s between any numbers already present in the sequence. Each iteration of the algorithm for generating \oddpart, run in reverse, involves inserting the odd numbers starting with $1$ into every gap between entries in the sequence produced by the previous iteration. If we take the odd positive integers mod $4$, however, they become a sequence of alternating $1$s and $3$s. 

To quickly see that this is the case, consider that the $k$th even number is $2k$, and thus the $k$th odd number is $2k-1$ (where $k$ is a positive integer). Now, every second even number \emd starting with 4 \emd is a multiple of $4$ and thus is evenly divisible by $4$. This means that every second odd number \emd starting with 3 \emd will be one less than a multiple of $4$. Any number that is one less than a multiple of $4$, taken mod $4$, will be $3$. So, every second odd positive integer, mod $4$, will be $3$.

All the other odd positive integers \emd starting with $1$ \emd taken mod $4$ will be $1$. After all, every second odd number (starting with $1$) will be $2$ less than the next odd number. But the next odd number will be $1$ less than a multiple of $4$, as just discussed. Therefore, every second odd number \emd starting with $1$ \emd will the $3$ less than a multiple of $4$. But any number that is $3$ less than a multiple of $4$, taken mod $4$, is $1$.

So, the sequence of odd numbers mod $4$ will be a sequence of alternating $1$s and $3$s. This means that by inserting the sequence of odd numbers into the sequence created by the previous iteration, the algorithm for generating \oddpart ~is inserting a sequence that is equivalent, mod $4$, to a sequence of alternating 1s and 3s. 

Furthermore, we are interpreting the numbers generated by both algorithms as quantities of $90$-degree turns to make at given points along a path. But making $4$ such turns is equivalent to making none. (If you turn left $90$-degrees $4$ times, you will be facing in the same direction you were before you turned.) This means that the result of making $k$ $90$-degree turns in a single direction is equivalent to making $k$ mod $4$ of those same turns. 

Therefore, whether we take \oddpart ~in its original form, or take it mod $4$, the same figure will be produced when its terms are interpreted as $90$-degree turns in a single direction. And that figure will be the Heighway Dragon.

\subsection{Conclusion}
The \levy Dragon is produced by \svpn{2} \emd the sequence of exponents of 2 in the prime factorizations of the positive integers. The Heighway Dragon is produced by \oddpart ~\emd the odd part of $n$ sequence, which is obtained by dividing the positive integers by $2^{\text{\vpn{2}}}$. There is a sense, then, in which every positive integer $n$ just is $2^{\text{\vpn{2}}}\times o(n)$. That is, it takes two fractal dragons to produce the sequence of positive integers.

\bibliographystyle{jis}
\bibliography{primefactors}
\end{document}